\documentclass{kms-j}

\newcommand{\publname}{This is a pre-print of an article published in J. Korean Math. Soc. (2021), v. \textbf{58}, p. 1485. The final authenticated version is available online at https://doi.org/10.4134/JKMS.j210119.}
\issueinfo{}
  {}
  {}
  {}
\pagespan{1}{}
\copyrightinfo{}
  {Korean Mathematical Society}

\usepackage{graphicx}
\allowdisplaybreaks

\theoremstyle{plain}
\newtheorem{theorem}{Theorem}[section]
\newtheorem{proposition}[theorem]{Proposition}
\newtheorem{lemma}[theorem]{Lemma}
\newtheorem{corollary}[theorem]{Corollary}

\theoremstyle{definition}
\newtheorem*{definition}{Definition}

\theoremstyle{remark}
\newtheorem{remark}[theorem]{Remark}

\newcommand{\rmd}{\mathrm{d}}

\begin{document}

\title{Curves orthogonal to a vector field in Euclidean spaces}


\author[Luiz C. B. da Silva]{Luiz C. B. da Silva}
\address[DA SILVA, L. C. B.]{Department of Physics of Complex Systems,\\ Weizmann Institute of Science, Rehovot 7610001, Israel}
\email{luiz.da-silva@weizmann.ac.il}  

\author[Gilson S. Ferreira Jr.]{Gilson S. Ferreira Jr. 
}
\address[FERREIRA JR., G. S.]{Department of Mathematics,\\
              Federal Rural University of Pernambuco, 52171-900, Recife-PE, Brazil}
\email{gilson.simoesj@ufrpe.br}
 
\subjclass[2010]{53A04, 53A05, 53C22} 
\keywords{Rectifying curve, geodesic, cone,  spherical curve, plane curve, slant helix}

\date{\today}

\maketitle

\begin{abstract}
A curve is rectifying if it lies on a moving hyperplane orthogonal to its curvature vector. In this work, we extend the main result of [Chen 2017, Tamkang J. Math. \textbf{48}, 209] to any space dimension: we prove that rectifying curves are geodesics on hypercones. We later use this association to characterize rectifying curves that are also slant helices in three-dimensional space as geodesics of circular cones. In addition, we consider curves that lie on a moving hyperplane normal to (i) one of the normal vector fields of the Frenet frame and to (ii) a rotation minimizing
vector field along the curve. The former class is characterized in terms of the constancy of a certain vector field normal to the curve, while the latter contains spherical and plane curves. Finally, we establish a formal mapping between rectifying curves in an $(m + 2)$-dimensional space and spherical curves in an $(m + 1)$-dimensional space. A curve is rectifying if it lies on a moving hyperplane orthogonal to its curvature vector. 
\end{abstract}

\maketitle

\section{Introduction}

In Euclidean space we may ask ``When does the position vector of a regular curve always lie orthogonal to a vector field?''. In other words, the problem consists in characterizing the curves $\alpha:I\to\mathbb{E}^{m+2}$ for which $\langle\alpha-p,\mathbf{V}\rangle=0$ in $I$, where $p$ is {constant} and $\mathbf{V}$ is a vector field along $\alpha$. Naturally, the answer will greatly depend on the properties of $\mathbf{V}$. For example, if $\alpha$ is a \emph{normal curve} (here $\mathbf{V}=\alpha'$), then the curve is spherical. On the other hand, if $\alpha$ is an \emph{osculating curve} (here $\mathbf{V}$ is the multinormal vector field, i.e., the last Frenet vector field from which we define the torsion \cite{Kuhnel2015}), then every osculating curve is a hyperplane curve. In the 2000s, Chen introduced the notion of a \emph{rectifying curve} in the three-dimensional (3d) Euclidean space by imposing that $\alpha$ always lies in its rectifying plane \cite{ChenMonthly2003}, i.e., it lies in the plane spanned by the tangent and binormal vectors. Rectifying curves have remarkable properties \cite{ChenBIMAS2005} and, in addition, they can be characterized as geodesics on a cone \cite{ChenTJM2017}. The notion of rectifying curves may be extended to higher dimensional Euclidean spaces \cite{CambieTJM2016,IlarslanTJM2008} by requiring $\alpha$ to lie in the (moving) hyperplane normal to its \emph{curvature vector} $\mathbf{k}=\kappa\, \mathbf{T}'/\Vert\mathbf{T}'\Vert$, where $\mathbf{T}=\alpha'/\Vert\alpha'\Vert$ is the unit tangent and $\kappa$ is the \emph{curvature function} of $\alpha$. Naturally, we can also consider curves orthogonal to one of the remaining vector fields of the Frenet frame (a problem originally proposed by Cambie et al. \cite{CambieTJM2016}) or, more generally, curves orthogonal to vector fields coming from frames distinct of Frenet, such as the so-called rotation minimizing (RM) frames \cite{BishopMonthly1975} (a problem investigated in 3d space by the first named author \cite{DaSilvaArXiv2017}). In addition, an equation relating the curvatures and torsion characterizing these special classes of curves has been obtained for rectifying curves, first in dimensions 3 and 4, \cite{ChenMonthly2003,KimHMJ1993} and \cite{IlarslanTJM2008}, respectively, and latter generalized to any dimension  \cite{CambieTJM2016}.

In this work, we extend the main result of Chen \cite{ChenTJM2017} to any space dimension: we prove that rectifying curves are geodesics on {hypercones} (Theorem \ref{ThrRecCurvAsConeGeod}). We later use this relation with {hypercones} to present a characterization of curves that are simultaneously rectifying and slant helices, i.e., curves whose curvature vector makes a constant angle with a fixed direction (a problem raised in 3d space by Deshmukh et al. \cite{AlghanemiFilomat2019}). Indeed, we show that in dimension 3 these curves are characterized as geodesics of circular cones (Theorem \ref{thr::SlantAndRectfCurves}). In higher dimensions, we {show} that {geodesics} of circular hypercones are slant helices (Corollary \ref{cor::SlantRectAsGeodCircCone}). In addition, we also consider curves that lie on a moving hyperplane normal to the $j$-th vector field of the Frenet frame and characterize them in terms of the constancy of a certain vector field normal to the curve, namely, the projection of the curve on the hyperplane spanned by the $(j+1)$-th, $(j+2)$-th, $\dots$, and multinormal vector fields of the Frenet frame (Theorem \ref{Theo::CharjRectCurves}). Later, by investigating the behavior of the coordinates of the curve with respect to a given orthonormal moving frame {(Lemma \ref{lemmaFrenetSystemForCoord})}, we establish a formal mapping between spherical and rectifying curves (Theorem \ref{TheoMapSpherCurveInRectCurve}). Finally, we characterize spherical and plane curves as those curves whose position vector is orthogonal to a rotation minimizing vector field (Theorem \ref{ThrCharRMrectifying}).
 
The remaining of this work is divided as follows. In Sect. \ref{secRectifyingCurves}, we study rectifying curves in Euclidean spaces. In Sect. \ref{sect::RectCurvAndSlntHlx}, we characterize those rectifying curves that are also slant helices. In Sect. \ref{sectJ-rect}, we investigate curves normal to a Frenet vector field. In Sect. \ref{sectMapSphrclAndRectfyngCrv}, we establish a map between spherical and rectifying curves and, finally, in Sect. \ref{secRMrectifyingCurves}, we consider curves normal to a rotation minimizing vector field.

\section{Rectifying curves in Euclidean spaces}
\label{secRectifyingCurves}

In this section, we generalize the main result of Chen \cite{ChenTJM2017} and show that rectifying curves in $\mathbb{E}^{m+2}$ are geodesics on {hypercones}. This characterization follows 
as a consequence of  Theorem \ref{thr::CharRectCurvesUsingTangentialComponent}, which is a generalization of Theorems 1 and 2 of \cite{ChenMonthly2003}. Such extensions already appeared in \cite{IlarslanTJM2008} for dimension 4 and in \cite{CambieTJM2016} for any dimension. The attentive reader will note that our proofs are similar to those of \cite{CambieTJM2016,ChenMonthly2003,IlarslanTJM2008}, but we included them here for the sake of completeness.  

Let $\alpha:I\to \mathbb{E}^{m+2}$ be a regular $C^2$ curve parameterized by the arc-length $s$, i.e., for all $s\in I$, $\langle\mathbf{T}(s),\mathbf{T}(s)\rangle=1$, where $\mathbf{T}(s)=\alpha'(s)$. We say that a $C^2$ regular curve $\alpha$ is \emph{rectifying} with vertex $p$ if $\langle\alpha(s)-p,\mathbf{k}(s)\rangle=0,$ where $\mathbf{k}(s)=\alpha''(s)$ is the curvature vector of $\alpha$ and $p$ is constant.

\begin{theorem}\label{thr::CharRectCurvesUsingTangentialComponent}
The following conditions are equivalent:
\begin{enumerate}
    \item The curve $\alpha(s)$ is rectifying.
    \item There exist constants $b$ and $c\in\mathbb{R}$ such that $\langle\alpha(s)-p,\mathbf{T}(s)\rangle=s+b$ and $\rho(s)\equiv\Vert\alpha(s)-p\Vert=\sqrt{s^2+2bs+c}$.
    \item There exist a reparameterization $t=t(s)$ and a unit velocity spherical curve $\beta:J\to\mathbb{S}^{m+1}(p,1)$ such that
    $$
        \alpha(t)=(a\sec t)\beta(t),
    $$
    where $a\in\mathbb{R}$ is a positive constant. (Note that $t$ is the arc-length of $\beta$.)
    \item The normal component of $\alpha(s)-p$ has constant length and $\rho(s)$ is a non-constant function. 
\end{enumerate}
\end{theorem}
\begin{proof}
$(1)\Leftrightarrow(2)$: Taking the derivative of $\langle\alpha(s)-p,\mathbf{T}(s)\rangle$ and using the definition of rectifying curves give
\begin{equation}
    \langle\alpha(s)-p,\mathbf{T}(s)\rangle' = \langle\mathbf{T}(s),\mathbf{T}(s)\rangle+\langle\alpha(s)-p,\mathbf{k}(s)\rangle=1.
\end{equation}
Thus, we conclude that $\langle\alpha(s)-p,\mathbf{T}(s)\rangle=s+b$ for some constant $b$. In addition,
\begin{equation}
    (\rho^2)'(s)=2\langle\alpha(s)-p,\mathbf{T}(s)\rangle=2s+2b\Rightarrow \exists\,c\in\mathbb{R}: \rho^2(s)=s^2+2bs+c.
\end{equation}

Conversely, if $\langle\alpha(s)-p,\alpha(s)-p\rangle=s^2+2bs+c$, then taking the derivative twice gives $1+\langle\alpha(s)-p,\mathbf{T}'(s)\rangle=1$, which implies $\langle\alpha(s)-p,\mathbf{T}'(s)\rangle=0$, i.e., $\alpha$ is a rectifying curve.  
\newline
$(2)\Leftrightarrow(3)$: First,  we write $\rho^2=(s+b)^2+a^2$, where $a^2=c-b^2>0$ (note $\rho^2>0$). Translating $s$, we may simply write $\rho^2=s^2+a^2$ and $\langle\alpha-p,\mathbf{T}\rangle=s$. Let us define the spherical curve $\beta(s)=\frac{1}{\rho(s)}(\alpha(s)-p)$. Then,
\begin{equation}
    \alpha(s)-p=\sqrt{s^2+a^2}\,\beta(s)\Rightarrow \mathbf{T}(s)=\frac{s}{\sqrt{s^2+a^2}}\beta(s)+\sqrt{s^2+a^2}\,\beta'(s).\label{eqDefSphrCurvFromRectifyinCurv}
\end{equation}
Since $\langle\beta,\beta'\rangle=0$, we deduce that $\Vert\beta'(s)\Vert=\frac{a}{s^2+a^2}$. The arc-length, $t$, of  $\beta$ is
\begin{equation}
    t = \int_0^s\frac{a}{u^2+a^2}\rmd u = \arctan\left(\frac{s}{a}\right)\Rightarrow s = a\tan t.
\end{equation}
Finally, substitution in Eq. \eqref{eqDefSphrCurvFromRectifyinCurv} leads to the desired result: $\alpha(t)-p=(a\sec t)\beta(t)$.

Conversely, if we have $\alpha(t)-p=(a\sec t)\beta(t)$, then it follows that $\alpha'(t)=(a\sec t)[\tan(t)\beta(t)+\beta'(t)]$. The arc-length parameter of $\alpha$ is $s=\int\Vert\alpha'(t)\Vert\rmd t=a\int\sec^2t\,\rmd t=a\tan t$. Finally,  $\rho^2(s)=\langle\alpha(s)-p,\alpha(s)-p\rangle=a^2\sec^2t=s^2+a^2$, which gives (2).
\newline
$(1)\Rightarrow(4)$: We can assume that (2) and (3) are valid [they are a consequence of (1)], from which follows that $\rho^2=\langle \alpha(t)-p,\alpha(t)-p\rangle=s^2+a^2=a^2\sec^2t$, $\langle\alpha-p,\mathbf{T}\rangle=s=a\tan t$, and 
\begin{equation}
 \alpha'(t)=(a\sec t)[\tan(t)\beta(t)+\beta'(t)]\Rightarrow \Vert \alpha'(t)\Vert=a\sec^2t.   
\end{equation}
The normal component $\alpha^N$ of $\alpha(t)-p=(a\sec t)\beta(t)$ is
\begin{equation}
\alpha^N(t)=(\alpha(t)-p)-\frac{\langle\alpha(t)-p,\alpha'(t)\rangle}{\Vert\alpha'(t)\Vert^2}\,\alpha'(t),
\end{equation}
which finally implies
\begin{eqnarray}
\langle\alpha^N(t),\alpha^N(t)\rangle &=& \langle\alpha(t)-p,\alpha(t)-p\rangle-\frac{\langle\alpha(t)-p,\alpha'(t)\rangle^2}{\Vert\alpha'(t)\Vert^2} \nonumber\\
 &=& \langle\alpha(t)-p,\alpha(t)-p\rangle-\langle\alpha(t)-p,\mathbf{T}(t)\rangle^2 \nonumber\\
& = & a^2\sec^2t-a^2\tan^2t=a^2.
\end{eqnarray}
$(4)\Rightarrow(1)$: Writing $\alpha-p=\langle\alpha-p,\mathbf{T}\rangle\,\mathbf{T}+\alpha^N$, where $\langle\alpha^N,\mathbf{T}\rangle=0$, and $C=\langle\alpha^N,\alpha^N\rangle$ constant, it follows
\begin{equation}
    \langle\alpha-p,\alpha-p\rangle = \langle\alpha-p,\mathbf{T}\rangle^2+\langle\alpha^N,\alpha^N\rangle=\langle\alpha-p,\mathbf{T}\rangle^2+C.
\end{equation}
Taking the derivative,
\begin{equation}
    2\langle\alpha-p,\mathbf{T}\rangle = 2\langle\alpha-p,\mathbf{T}\rangle\Big(\langle \mathbf{T},\mathbf{T}\rangle+\langle\alpha-p,\mathbf{T}'\rangle\Big)\Rightarrow 1=1+\langle\alpha-p,\mathbf{k}\rangle,
\end{equation}
where {we} used that $\rho$ non-constant implies $\langle\alpha-p,\mathbf{T}\rangle\not=0$. Finally, we deduce that $\langle\alpha-p,\mathbf{T}'\rangle=0$, i.e., $\alpha$ is a rectifying curve. 
\end{proof}

A {hypercone} $\mathcal{C}^{m+1}(p)$ in $\mathbb{E}^{m+2}$ with vertex at $p$ can be parameterized in terms of a spherical submanifold as
\begin{equation}\label{eq::SphericalRepOfCones}
    C_{\beta}(t_1,\dots,t_{m},u) = u\beta(t_1,\dots,t_{m}),
\end{equation}
where $\beta:(t_1,\dots,t_{m})\mapsto\mathbb{S}^{m+1}(p,1)$ is a regular hypersurface. For a given point $\mathbf{t}_0=(t_1,\dots,t_{m})$, the straight lines $c(t)=C_{\beta}(\mathbf{t}_0,t)$ are geodesics of the {hypercone}, these are the so-called \emph{rulings}. If $\beta$ parameterizes a great sphere, i.e., the intersection of $\mathbb{S}^{m+1}(p,1)$ with a hyperplane passing through $p$, the corresponding {hypercone} is just a hyperplane, whose geodesics are all straight lines. Thus, in the following we assume that this is not the case. The next theorem characterizes the remaining geodesics on {a hypercone} as rectifying curves and generalizes the main result in Ref. \cite{ChenTJM2017}. In fact, we have the

\begin{theorem}\label{ThrRecCurvAsConeGeod}
A regular $C^2$ curve $\alpha:I\to\mathbb{E}^{m+2}$ is rectifying with vertex $p$ if and only if it is a geodesic on a {hypercone} $\mathcal{C}^{m+1}(p)$ which is not a ruling.
\end{theorem}
\begin{proof}
Let $\alpha(t)=u(t)\beta(t_1(t),\dots,t_{m}(t))\equiv u(t)\beta(t)$ be a geodesic on $\mathcal{C}^{m+1}(p)$ with $\beta(t)\in\mathbb{S}^{m+1}(p,1)$ a unit speed curve. (Note, $\alpha$ is not a ruling.) We have $\alpha'(t)=u'(t)\beta(t)+u(t)\beta'(t)$ and, therefore, the length functional of $\alpha$, which is a function of $t,u$, and $u'$ only, is given by
$L(t,u,u') =\int E\,\rmd t= \int\sqrt{u^2+u'^2}\,\rmd t$. The corresponding Euler-Lagrange equation is
\begin{equation}
    \frac{\partial E}{\partial u}-\frac{\rmd}{\rmd t}\frac{\partial E}{\partial u'}=0\Rightarrow uu''-2u'\,^2-u^2=0. 
\end{equation}

The general solution is of the form $u(t)=a\sec(t+b)$ for some constants $a,b\in\mathbb{R}$. Indeed, defining $v(u)=\rmd u/\rmd t$ leads to $uv(u)v'(u)-2v(u)^2-u^2=0$ and, dividing by $u/2$, $2v(u)v'(u)-4v(u)^2/u=2u$. We may now define $w=v^2$ and, therefore, $w'(u)-4w(u)/u=2u$. Multiplying this equation by $\mu=1/u^4$ (integrating factor), we have $(w/u^4)'(u)=2/u^3=-(1/u^2)'$. Then, there exists a constant $c$, such that $c^2=w/u^4+1/u^2\Rightarrow c^2\,u^4=v^2+u^2=u'(t)^2+u^2$ or, equivalently,  $(c\,u)^4=(c\,u')^2+(c\,u)^2$, whose general solution is a secant function. Then, every geodesic of a {hypercone} is a rectifying curve.

Conversely, from Theorem \ref{thr::CharRectCurvesUsingTangentialComponent}, it follows that a rectifying curve can be written as $\alpha(t)=a\sec(t)\beta(t)$, where $\beta:I\to \mathbb{S}^{m+1}(p,1)$ is a unit speed curve.  Using the reasoning above, $\alpha$ is a geodesic of the 2-cone $\Sigma^2:(u,s)\mapsto p+u(\alpha(s)-p)$. Thus, $\alpha''(s)$ is orthogonal to $T_{\alpha(s)}\Sigma^2$. Now, let $V_2(s),\dots,V_{m}(s)$ be unit vector fields orthogonal to $\alpha''$ and $\Sigma^2$. We may use these vector fields to build a hypercone such that $\alpha$ is a geodesic in it. Indeed, define
\begin{equation}
\mathcal{C}^{m+1}(p):(u,s_1,\dots,s_m)\mapsto p+u\,[\alpha(s_1)+\sum_{i=2}^ms_iV_i(s_1)-p].    
\end{equation}
Note that $\alpha(s)$ has coordinates $(u,s_1,\dots,s_m)=(1,s,0,\dots,0)$. Then, the tangent vectors along $\alpha$ are 
\begin{equation}
\left\{
\begin{array}{ccl}
     \partial_u\vert_{\alpha} &=& \Big(\alpha(s_1)+\sum_{j=2}^{m}s_jV_{j}(s_1)-p\Big)\Big\vert_{\alpha(s)}=\alpha(s)-p \\[5pt]
     \partial_{s_1}\vert_{\alpha} &=& \Big(u\alpha'(s_1)+u\sum_{j=2}^{m}s_jV_{j}'(s_1)\Big)\Big\vert_{\alpha(s)}=\alpha'(s) \\[5pt]
     \partial_{s_j}\vert_{\alpha} &=& V_{j}(s),\,j\in\{2,\dots,m\}\\
\end{array}
\right..
\end{equation}
By construction, $\alpha''$ {is orthogonal to} $\partial_{s_j}\vert_{\alpha}$ {for all $j\in\{2,\dots,m\}$} and, since $\alpha$ is rectifying and parameterized by arc-length, we also have that $\alpha''$ is orthogonal to $\partial_{u}\vert_{\alpha}$ and  $\partial_{s_1}\vert_{\alpha}$. Therefore, $\alpha''$ is parallel to the normal of $\mathcal{C}^{m+1}(p)$ and, consequently, $\alpha$ is a geodesic.
\end{proof}

We now provide an alternative proof for the characterization of geodesics on a {hypercone}. The strategy of the proof is similar to that found in the study of rectifying curves in the 3d sphere and hyperbolic space \cite{LucasJMAA2015,LucasMJM2016}.

\begin{proof}[Alternative proof of Theorem \ref{ThrRecCurvAsConeGeod}.] Given $\alpha:I\to\mathcal{C}^{m+1}(p)$, it follows by definition of a {hypercone} that the straight line $X_s(u)=p+u(\alpha(s)-p)$ is a curve in $\mathcal{C}^{m+1}(p)$ for every $s\in I$. Consequently, the velocity vector  $X_s'(u)=\alpha(s)-p$ is tangent to $\mathcal{C}^{m+1}(p)$. Thus, if $\alpha$ is a geodesic, we must have $\langle\alpha'',\alpha-p\rangle=0$, i.e., $\alpha$ is a rectifying curve. 

Conversely, let $\alpha$ be a rectifying curve centered at $p$. Now, consider the 2-cone $\Sigma^2:X(u,s)=p+u(\alpha(s)-p)$, $u\not=0$ and $s\in I$. By hypothesis, $\partial_uX=\alpha-p$ is orthogonal to $\alpha''$. On the other hand, since $\langle\alpha',\alpha'\rangle=1\Rightarrow\langle\alpha',\alpha''\rangle=0$ and $\alpha'=\frac{1}{u}\partial_sX$, we have $\langle\partial_sX,\alpha''\rangle=0$. Thus, we conclude that $\alpha''$ is normal to $\Sigma^2$, $\alpha''\in\Gamma(N\Sigma^2)$. Therefore, $\alpha$ is a geodesic of the 2-cone $\Sigma^2$. To conclude the proof, i.e., show that $\alpha$ is a geodesic of a hypercone, we may employ the same strategy used in the end of the first proof of Theorem \ref{ThrRecCurvAsConeGeod}. 
\end{proof}

\begin{remark}
A careful examination of the proofs of Theorem \ref{ThrRecCurvAsConeGeod} reveals that the {hypercone} containing a rectifying curve may be not unique. (Uniqueness is only assured for 2-cones.) Then, it would be interesting to ask whether the cones sharing a common rectifying curve have any special geometric property. 
\end{remark}

\section{Rectifying curves and slant helices}
\label{sect::RectCurvAndSlntHlx}

A curve is a \emph{slant helix} if its curvature vector makes a constant angle with a fixed direction \cite{IzumiyaTJM2004}. In this section, we are interested in characterizing those rectifying curves that are also slant helices. We mention that a characterization of rectifying slant helices in terms of their curvature and torsion was established by Altunkaya et al. \cite{AltunkayaKJM2016}. (This problem has been {also} recently raised by Deshmukh et al. \cite{AlghanemiFilomat2019}.) Here, we are going to provide a geometric answer to this question in terms of the spherical curve associated with the cone that contains a given rectifying curve as a geodesic. The strategy will consist in taking into account that \emph{constant angle surfaces} (also known as \emph{helix surfaces}), i.e., surfaces whose unit normal makes a constant angle with a fixed direction, are characterized by the fact their geodesics are slant helices \cite{LucasBBMS2016}. Then, we may characterize curves that are simultaneously rectifying and slant helices by determining the cones of constant angle (Proposition \ref{prop::HelixConesAreCircular}) which finally leads to the characterization of rectifying slant helices as geodesics of circular cones (Theorem \ref{thr::SlantAndRectfCurves}). In higher dimensions we show that geodesics of circular hypercones are slant helices (Corollary \ref{cor::SlantRectAsGeodCircCone}) while the converse remains an open problem.

\begin{proposition}\label{prop::HelixConesAreCircular}
A cone $\mathcal{C}_{\beta}^2(p)\subset\mathbb{E}^3$ makes a constant angle with a fixed direction if and only if it is circular. In addition, the fixed direction coincides with the axis of the circular cone.
\end{proposition}
\begin{proof}
The unit normal of $\mathcal{C}_{\beta}^2$ along $\beta$ is the normal to $\beta$ with respect to the unit sphere, namely $\xi\vert_{\beta}=\beta\times\beta'$. On the remaining points of $\mathcal{C}_{\beta}^2$, the normal is obtained through parallel transport along the rulings. In addition, the Frenet equations of $\beta$ on the sphere are $\nabla_{\beta'}\beta'=\kappa_g\xi$ and $\xi'=-\kappa_g\beta'$, where $\nabla_{\beta'}\beta'\equiv\mbox{Proj}_{T_{\beta}\mathbb{S}^2}(\beta'')=\beta''+\beta$ and $\kappa_g$ is the geodesic curvature of $\beta$ with respect to the unit sphere.

The spherical curve associated with  a circular cone is a small circle described by the equation $\langle\beta,\mathbf{d}\rangle=\mbox{const.}$, which gives $\langle\beta',\mathbf{d}\rangle=0$. Then, taking the derivative of $f(t)=\langle\xi,\mathbf{d}\rangle$ leads to $f'=-\kappa_g\langle\beta',\mathbf{d}\rangle=0$. Therefore, $f=\mbox{const.}$ and, consequently, a circular cone makes a constant angle with a fixed direction. 

Conversely, assume that the cone $\mathcal{C}_{\beta}^2$ makes a constant angle with the fixed direction $\mathbf{d}$, $\langle\xi,\mathbf{d}\rangle=\mbox{constant}$. We have $0=\langle\xi',\mathbf{d}\rangle=-\kappa_g\langle\beta',\mathbf{d}\rangle$. If $\kappa_g=0$, then $\beta$ is a great circle and $\mathcal{C}_{\beta}^2$ is a plane, which is a helix surface. On the other hand, if $\kappa_g\not=0$, we deduce that $\langle\beta',\mathbf{d}\rangle=0\Rightarrow \langle\beta,\mathbf{d}\rangle=\mbox{const.}$ and, therefore, $\beta$ is a small circle and $\mathcal{C}_{\beta}^2$ is a circular cone.
\end{proof}

\begin{theorem}\label{thr::SlantAndRectfCurves}
A rectifying curve $\alpha:I\to\mathbb{E}^3$ is a slant helix if and only if it is the geodesic of a circular cone.
\end{theorem}
\begin{proof}
The principal normal of a rectifying curve coincides with the cone normal since it is a geodesic. Thus, if the associated cone is circular, the curve will make a constant angle with a fixed direction, the axis of the cone. Therefore, any circular rectifying is a slant helix.

Conversely, if $\alpha$ is a rectifying  slant helix, then the unit normal of the cone $\mathcal{C}^2(p)$ containing $\alpha$ has to make a constant angle with a fixed direction along $\alpha$. For cones, the unit normal is parallel transported along the rulings and, consequently, the portion of $\mathcal{C}^2(p)$ given by $r(u,s) = p + u (\alpha(s)-p)$, $u \in [0,1]$,
should be a circular cone according to Proposition \ref{prop::HelixConesAreCircular}. Thus, a rectifying curve which is also a slant helix has to be circular rectifying and, in addition, the fixed direction is nothing but the axis of the corresponding circular cone.
\end{proof}

In \cite{LucasBBMS2016} it is noted that the geodesics of circular cones provide examples of rectifying slant helices. We just showed that this is a characteristic property. Now, we address the same problem in $\mathbb{E}^{m+2}$. Mimicking Eq. (\ref{eq::SphericalRepOfCones}), we can define {\emph{$n$-cones}} by taking $\beta$ as the parameterization of an $(n-1)$-dimensional submanifold of the unit sphere $\mathbb{S}^{m+1}(p,1)$. A $n$-cone is said to be \emph{circular} if $\langle\beta,\mathbf{d}\rangle=\mbox{const.}$ for some fixed vector $\mathbf{d}$.

\begin{proposition}
A circular hypercone $\mathcal{C}_{\beta}^{m+1}(p)\subset\mathbb{E}^{m+2}$ makes a constant angle with a fixed direction. In addition, the fixed direction coincides with the axis of the circular {hypercone}.
\end{proposition}
\begin{proof}
First, note that the unit normal $\xi$ of $\mathcal{C}_{\beta}^{m+1}$ along $\beta(t_1,\dots,t_m)$ has to be the normal to $\beta$ with respect to the unit sphere. On the remaining points of $\mathcal{C}_{\beta}^{m+1}$, the normal is obtained through parallel transport along the rulings in $\mathbb{E}^{m+2}$. In addition, if $\nabla$ is the covariant differentiation in $\mathbb{S}^{m+1}$,  $(\nabla_XY)(p)=\frac{\partial Y}{\partial X}\vert_p+\langle X,Y\rangle\,p$, we may conclude that $\nabla_{\partial_i}\xi=\frac{\partial \xi}{\partial t_i}$, where $\partial_i$ is the velocity vector of the $i$-th coordinate curve $t_i\mapsto\beta(t_1,\dots,t_i,\dots,t_m)$.

The spherical submanifold associated with  a circular hypercone is a small sphere described by an equation $\langle\beta,\mathbf{d}\rangle=\mbox{const.}$, which gives $\langle\partial\beta/\partial t_i,\mathbf{d}\rangle=0$ for all $i\in\{1,\dots,m\}$. Taking the derivative of the angle function $f(t_1,\dots,t_m)=\langle\xi,\mathbf{d}\rangle$ leads to $\partial f/\partial t_i=\langle\nabla_{\partial_i}\xi,\mathbf{d}\rangle=0$, where we used that $\nabla_{\partial_i}\xi$ has to be a tangent vector to $\beta$ in $\mathbb{S}^{m+1}$. Thus, $f=\mbox{const.}$ and, therefore, a circular hypercone makes a constant angle with a fixed direction.
\end{proof}

\begin{corollary}\label{cor::SlantRectAsGeodCircCone}
A circular rectifying curve, i.e., a geodesic of a circular hypercone, is also a slant helix.
\end{corollary}

In higher dimensions, the converse of the above result is subtler. If a rectifying curve $\alpha$ is also a slant helix, then $\alpha''$ coincides with the hypercone normal $\xi$ and, in addition, it is straightforward to conclude that $\xi$ makes a constant angle with a fixed direction along the 2-cone $\Sigma_{\alpha}^2:(u,s)\mapsto p+u(\alpha(s)-p)$. (Note that any hypercone containing $\alpha$ should necessarily contain $\Sigma_{\alpha}^2$.) Then, the challenge to establish a converse is to show that it is possible to find a hypercone $\mathcal{C}^{m+1}_{\beta}$ whose associated spherical submanifold $\beta$ is a small sphere.

\section{Curves normal with respect to a Frenet vector field}
\label{sectJ-rect}

It is known that rectifying curves can be characterized in terms of the constancy of the length of its normal component \cite{CambieTJM2016,ChenMonthly2003} [see Theorem \ref{thr::CharRectCurvesUsingTangentialComponent}, item (4)]. The problem of characterizing curves normal to one of the Frenet vectors was first proposed by Cambie et al. \cite{CambieTJM2016}: here we shall call a curve \emph{$j$-rectifying} if its position vector is orthogonal to the $j$-th Frenet vector field. In this section we provide a characterization for $j$-rectifying curves in terms of the constancy of a certain normal component (Theorem \ref{Theo::CharjRectCurves}), which then generalizes the characterization of rectifying curves, or 1-rectifying in our notation. First, we need some preliminaries results.

Let $\alpha:I\to \mathbb{E}^{m+2}$ be a regular curve parameterized by arc-length. We say that $\alpha$ is a \emph{twisted curve} if it is of class $C^{m+2}$ and $\{\alpha'(s),\alpha''(s),\dots,\alpha^{(m+2)}(s)\}$ is linearly independent for all $s\in I$ \cite{Kuhnel2015}. We may associate with a twisted curve its Frenet frame $\{\mathbf{T},\mathbf{N}_1,\dots.\mathbf{N}_m,\mathbf{B}\}$ whose equations of motion in $\mathbb{E}^{m+2}$ are
\begin{equation}\label{frenetsystem}
    \left\{
    \begin{array}{ccc}
         \mathbf{T}' & = & \kappa_0 \mathbf{N}_1  \\
         \mathbf{N}_i' & = & -\kappa_{i-1} \mathbf{N}_{i-1} + \kappa_i\mathbf{N}_{i+1} \\
         \mathbf{B}' & = & - \kappa_{m}\mathbf{N}_{m}\\
    \end{array}
    \right.,\,i\in\{1,\dots,m\},
\end{equation}
where $\mathbf{N}_0=\mathbf{T}$ is the unit tangent whose derivative gives the curvature function $\kappa=\kappa_0$ and $\mathbf{N}_{m+1}=\mathbf{B}$ is the \emph{multinormal vector} whose derivative gives the torsion $\tau=\kappa_{m}$. In analogy to what happens in 3d, a hyperplane curve is characterized by $\tau\equiv0$. Moreover, if $\alpha$ is twisted, then $\kappa_i\not=0$ and $\tau\not=0$.

\begin{definition}
We say that $\alpha$ is a \emph{$j$-rectifying curve}, $j\in\{0,\dots,m+1\}$, when
\begin{equation}
\forall\,s\in I,\,\langle\alpha(s)-p,\mathbf{N}_j(s)\rangle=0\Rightarrow \alpha-p=\sum_{i=0,\, i\not=j}^{m+1}A_i(s)\mathbf{N}_i(s),
\end{equation}
where $A_{i}(s) = \langle\alpha(s)-p, N_{i}(s) \rangle$.
\end{definition}
Note that for $j=0,1$, and $m+1$ we obtain the normal, rectifying, and osculating curves, respectively. Thus, it remains to investigate  the cases where $j\in\{2,\dots,m\}$.

\begin{lemma}\label{lemmaFrenetSystemForCoord}
Let $\alpha$ be \emph{any} $C^{2}$ regular curve and $\{\mathbf{V}_0=\mathbf{T},\mathbf{V}_1,\dots,\mathbf{V}_{m+1}\}$ be \emph{any} orthonormal moving frame along $\alpha$ whose equations of motion are
\begin{equation}\label{eqODEsForMovFrame}
    \mathbf{V}_i'(s)=\sum_{j=0}^{m+1}k_{ij}(s)\mathbf{V}_j(s),\,\mbox{ where }k_{ij}=-k_{ji}.
\end{equation}
If we write $\alpha(s)-p=\sum_{i=0}^{m+1}A_i(s)\mathbf{V}_i(s)$,
then the coordinate functions $\{A_i\}_{i=0}^{m+1}$ satisfy the system of equations
\begin{equation}\label{eqODEsForCoordWithRespetMovFrame}
         A_{0}'(s) = 1+ \sum_{j=0}^{m+1} k_{0j}(s) A_{j}(s)   \mbox{ and }
         A_{i}'(s) = \sum_{j=0}^{m+1} k_{ij}(s) A_{j}(s),
\end{equation}
where $i\in \{1,\dots,m,m+1\}$. {Conversely, assume that Eq. \eqref{eqODEsForCoordWithRespetMovFrame} is satisfied for some functions $\{A_j\}_{j=0}^{m+1}$ and some prescribed coefficients $k_{ij}$, $i,j\in\{0,\dots,m+1\}$, such that $k_{ij}=-k_{ji}$ and $\kappa^2\equiv k_{01}^2+\dots+k_{0,m+1}^2>0$. Then, we can integrate the Frenet-like system of equations \eqref{eqODEsForMovFrame} to obtain a regular curve $\alpha=\int V_0(s)\rmd s$ equipped with a moving frame $\{\mathbf{V}_j\}$ (fundamental theorem of curves) and such that the functions $\{A_j\}_{j=0}^{m+1}$ are the coordinates of $\alpha$ with respect to the moving frame $\{\mathbf{V}_j\}$. Finally, the derivative of the distance function $\rho=\Vert \alpha-p\Vert$ and the tangential coordinate, $A_0$, are related by $(\rho^2)'(s)=2A_0(s)$.}
\end{lemma}
\begin{proof}
For $i=0$, we have $A_{0}'=1+\langle\alpha-p,\sum_{i=0}^{m+1} k_{0j}\mathbf{V}_j\rangle=1+ \sum_{i=0}^{m+1} k_{0j} A_{j}.$ Now, for $i\in\{1,\dots,m+1\}$, we have $A_{i}'=0+\langle\alpha-p,\sum_{j=0}^{m+1}k_{ij}(s)\mathbf{V}_j\rangle=\sum_{i=0}^{m+1} k_{ij} A_{j}$. In short, the coordinates functions  satisfy the system \eqref{eqODEsForCoordWithRespetMovFrame}.

{Conversely, assume the validity of the system of equations \eqref{eqODEsForCoordWithRespetMovFrame}. Let $\alpha=\int V_0(s)\rmd s$ be the curve obtained by integrating \eqref{eqODEsForMovFrame} and define the curve $\beta(s)=\alpha(s)-\sum_{i=0}^{m+1}A_i(s)\mathbf{V}_i(s)$. Taking the derivative of $\beta$ gives}
\begin{eqnarray}
    \beta' & = & \alpha'-\sum_{i=0}^{m+1}(A_i'\mathbf{V}_i+A_i\mathbf{V}_i')\nonumber\\
           & = & \mathbf{V}_0-(1+ \sum_{j=0}^{m+1} k_{0j} A_{j})\mathbf{V}_0-\sum_{i=1}^{m+1}\sum_{j=0}^{m+1} k_{ij} A_{j}\mathbf{V}_i-\sum_{i=0}^{m+1}A_i\sum_{j=0}^{m+1} k_{ij}\mathbf{V}_j\nonumber\\
           & = & -\sum_{i=0}^{m+1}\sum_{j=0}^{m+1} k_{ij} A_{j}\mathbf{V}_i-\sum_{i=0}^{m+1}\sum_{j=0}^{m+1} k_{ij}A_i\mathbf{V}_j=0,
\end{eqnarray}
{where for the last equality we replaced the indices $i,j$ in the second sum and then used that $k_{ij}=-k_{ji}$. Therefore, we conclude that the functions $A_0,\dots,A_{m+1}$ are the coordinates of $\alpha-p$, for some $p$ constant, with respect to the corresponding moving frame $\mathbf{V}_0=\alpha',\dots,\mathbf{V}_{m+1}$.}

{It remains to} investigate $\rho$. First, note that $k_{ij}=-k_{ji}$ follows as a result of the orthonormality of $\{\mathbf{V}_i\}$. {Now, using} that $\rho^2=\sum_{i=0}^{m+1}A_i^2$, we finally have
\begin{eqnarray}
(\rho^2)' & = & 2A_0A_0'+2\sum_{i=1}^{m+1}A_iA_i'= 2A_0(1+\sum_{i=0}^{m+1} k_{0j} A_{j})+2\sum_{i=1,j=0}^{m+1}k_{ij}A_iA_{j}\nonumber\\
 & = & 2A_0+2\sum_{i=0,j=0}^{m+1}k_{ij}A_iA_{j}= 2A_0+\sum_{i<j}(k_{ij}+k_{ji})A_iA_{j}=2A_0.
\end{eqnarray}
\end{proof}

\begin{remark}
{Usually, the fundamental theorem of curves is presented and proved for the Frenet frame, i.e., by prescribing $\{\kappa_0>0,\kappa_1,\dots,\kappa_{m}\}$ we can integrate the Frenet system \eqref{frenetsystem} to obtain a regular curve, up to rigid motions, whose curvatures are precisely $\{\kappa_j\}$, e.g., see \cite{Kuhnel2015}. (If we drop the restriction $\kappa_0>0$, we still have existence but we can no longer guarantee uniqueness up to rigid motions.) It is not difficult to see that the same proof can be adapted for any other moving frame with curvature coefficients $k_{ij}$ as in Eq. \eqref{eqODEsForMovFrame}. For example, another possibility is given by the so-called rotation minimizing frames \cite{BishopMonthly1975,Etayo2016} that will be used in Section \ref{secRMrectifyingCurves}.}
\end{remark}

Equipping a curve with its Frenet frame and, in addition, taking into account that a curve is $j$-rectifying when its $j$-th coordinate function $A_j$ vanishes, the following result holds.

\begin{corollary}\label{Cor::FrenetSystemForCoord}
Let $\alpha$ be \emph{any} $C^{m+2}$ regular curve and $\{A_i\}_{i=0}^{m+1}$ be the coordinate functions with respect to the Frenet frame $\{\mathbf{N}_i\}_{i=0}^{m+1}$. Then, the coefficients $\{A_i\}$ satisfy the Frenet-like system of equations
\begin{equation}\label{eqODEsForCoordjRectifying}
   \left\{
    \begin{array}{ccc}
         A_{0}'(s) & = & 1+ \kappa(s) A_{1}(s)   \\[3pt]
         A_{i}'(s) & = & -\kappa_{i-1}(s)A_{i-1}(s)+\kappa_{i}(s)A_{i+1}(s) \\[3pt]
         A_{m+1}'(s) & = & -\tau(s) A_{m}(s)\\
    \end{array}
    \right.,\,i\in \{1,\dots,m\}.
\end{equation}
Moreover, if $\alpha$ is a $j$-rectifying curve, then we have the additional equations
\begin{equation}
    A_{j-1}'=-\kappa_{j-2}A_{j-2},\,A_{j+1}'=\kappa_{j+1}A_{j+2},\mbox{ and }-\kappa_{j-1}A_{j-1}+\kappa_{j}A_{j+1}=0.
\end{equation}
\end{corollary}

\begin{lemma}\label{LemNoJandJplus1RectCurv}
Let $\alpha:I\to\mathbb{E}^{m+2}$ be a regular twisted curve and $\{A_i\}_{i=0}^{m+1}$ be the coordinate functions with respect to its Frenet frame. Then, $\alpha$ can not be simultaneously a $j$- and a $(j+1)$-rectifying curve for any $j$.
\end{lemma}
\begin{proof}
Assume that $\alpha$ is both $j$- and $(j+1)$-rectifying for some $j$. Then, it follows that $0=A_j'=-\kappa_{j-1}A_{j-1}+\kappa_{j}A_{j+1}$. Now, since $\alpha$ is also $(j+1)$-rectifying and  $\kappa_{j-1}\not=0$ ($\alpha$ twisted), we have $A_{j-1}=0$. Thus, $\alpha$ is also $(j-1)$-rectifying. By recursion, we would deduce that $\alpha$ is 1-, 2-,$\dots$, and $(j-1)$-rectifying. (Note that $A_0'=1\Rightarrow A_0=s+b$.) Analogously, from $A_{j+1}=0$, we also have $0=-\kappa_{j}A_{j}+\kappa_{j+1}A_{j+2}=\kappa_{j+1}A_{j+2}$ and, consequently, $A_{j+2}=0$. In short, if $\alpha$ were $j$- and $(j+1)$-rectifying for some $j$ we would deduce that all $A_i$ vanish except for $A_0$, which implies $\alpha=p+(s-a)\mathbf{T}$. Thus $\alpha$ would be a straight line, which  is {not} twisted. 
\end{proof}

Now, we provide a proof for the main theorem of this section characterizing $j$-rectifying curves, which should be thought as the generalization of item (4) of  Theorem \ref{thr::CharRectCurvesUsingTangentialComponent} to this new context.

\begin{theorem}\label{Theo::CharjRectCurves}
Let $\alpha:I\to\mathbb{E}^{m+2}$ be a regular curve of class $C^{m+2}$. Then, $\alpha$ is $j$-rectifying, {$j\in\{0,\dots,m\}$,} if and only if the normal vector field $$\alpha^{N_j}\equiv\sum_{i=j+1}^{m+1}\langle\alpha-p,\mathbf{N}_i\rangle\mathbf{N}_i$$ has constant length.
\end{theorem}
\begin{proof}
Let $\alpha$ be $j$-rectifying, i.e., $A_j=0$. Since $\rho_{j}^2\equiv\langle\alpha^{N_j},\alpha^{N_j}\rangle=\sum_{i=j+1}^{m+1}A_i^2$, taking the derivative gives 
\begin{eqnarray*}
(\rho_{j}^2)' &=& 2\sum_{i=j+1}^{m}(-\kappa_{i-1}A_{i-1}A_i+\kappa_iA_iA_{i+1})+2A_{m+1}A_{m+1}'\\
&=& 2(-\kappa_{j}A_{j}A_{j+1}+\tau A_mA_{m+1})-2\tau A_{m}A_{m+1}=0.
\end{eqnarray*}
Therefore, $\alpha^{N_j}$ has constant length. 

Conversely, let $\rho_{j}$ be constant. We may assume, without loss of generality, that $A_{j+1}\not\equiv0$, otherwise $\rho_{j}=\rho_{j+1}$ and we can exchange $j$ and $j+1$ (see Lemma \ref{LemNoJandJplus1RectCurv}). We can write $\alpha-p$ as
\begin{equation}
    \alpha(s)-p=\sum_{i=0}^jA_i\mathbf{N}_i+\alpha^{N_j}\Rightarrow \rho^2 = \sum_{i=0}^jA_i^2+\rho_{j}^2.
\end{equation}
Taking the derivative, and using Corollary \ref{Cor::FrenetSystemForCoord},
\begin{eqnarray}
    (\rho^2)' & = & 2\sum_{i=0}^jA_iA_i'+0= 2A_0(1+\kappa A_1)+2\sum_{i=1}^j(-\kappa_{i-1}A_{i-1}A_i+\kappa_{i}A_iA_{i+1})\nonumber\\
    & = & 2A_0 +2\kappa A_0A_1+2(-\kappa A_0A_1+\kappa_jA_jA_{j+1}) \nonumber\\
    &=& 2A_0+2\kappa_jA_jA_{j+1}.
\end{eqnarray}
Since $(\rho^2)'=2A_0$ (Lemma \ref{lemmaFrenetSystemForCoord}), it follows that $\kappa_jA_jA_{j+1}=0$ and, consequently, $A_j=0$. In other words, $\alpha$ is a $j$-rectifying curve.
\end{proof}

\begin{remark}
The definition of $j$-rectifying curves only requires the curve to be of class $C^{j+1}$ since the Frenet frame is defined in such a way  that $V_j\equiv {\mathrm{span}}   \{\mathbf{T},\mathbf{N}_1,\dots,\mathbf{N}_j\}= {\mathrm{span}} \{\alpha',\alpha'',\dots,\alpha^{(j+1)}\}$ \cite{Kuhnel2015}. Therefore, once we equip a $j$-rectifying curve with the first $j+1$ Frenet vectors, we could later choose any set of $m-j+1$ orthonormal vector fields spanning $V_j^{\perp}$ to complete a frame along $\alpha$ and then provide a proof entirely analogous to the above. 
\end{remark}

\section{A formal correspondence between spherical and rectifying curves}
\label{sectMapSphrclAndRectfyngCrv}

In $\mathbb{E}^3$, in addition to the characterization of rectifying curves in terms of $\Vert\alpha^N\Vert=\mbox{constant}$, Chen showed that $\alpha$ is rectifying if and only if $\frac{\tau}{(s+b)\kappa}=\frac{1}{a}$, for some constants $a$ and $b$ \cite{ChenMonthly2003}. Item (iv) of Chen's Theorem 1 \cite{ChenMonthly2003} was later extended to higher dimensional spaces in Theorem 4.1 of \cite{CambieTJM2016}, see also \cite{IlarslanTJM2008} for a proof in $4d$, involving the remaining curvature functions. In this section, we show that these equations for the curvatures and torsion of a rectifying curve in $\mathbb{E}^{m+2}$ allow us to establish a correspondence with spherical curves in $\mathbb{E}^{m+1}$ (see Theorem \ref{TheoMapSpherCurveInRectCurve}). To illustrate this, {if $\kappa_0=\kappa$ and $\kappa_1=\tau$ denote the curvature and torsion of a rectifying curve in $\mathbb{E}^3$, then we can establish the following correspondence with circles in $\mathbb{E}^2$ of curvature $k_0=1/a$ given by}
\begin{equation}
 {    (\kappa_0,\kappa_1) \mapsto k_0(\kappa_0,\kappa_1)\equiv\frac{\kappa_1}{(s+b)\kappa_0}=\frac{1}{a}.
 }
\end{equation}
{The existence of the associated curve with curvature $k_0=1/a$ is then guaranteed by the fundamental theorem of curves in $\mathbb{E}^2$.}

Analogously,  rectifying curves in $\mathbb{E}^4$ 
are characterized by the equation \cite{CambieTJM2016,IlarslanTJM2008} (note our notation is a bit different: $\kappa_0=\kappa_1,\kappa_1=\kappa_2$, and $\tau=\kappa_3$) 
\begin{equation}\label{eq::DEqFor4dRectCurves}
   \frac{(s+b)\kappa_0}{\kappa_1}\tau+\frac{\rmd}{\rmd s}\left\{\frac{1}{\tau}\frac{\rmd}{\rmd s}\left[\frac{(s+b)\kappa_0}{\kappa_1}\right]\right\}=0. 
\end{equation}
Consequently, we may establish a formal correspondence {of rectifying curves in $\mathbb{E}^4$ with  spherical curves in $\mathbb{E}^3$ of curvature $k_0$ and torsion $k_1$ given by}
\begin{equation}
 {
 (\kappa_0,\kappa_1,\tau) \mapsto (k_0(\kappa_0,\kappa_1,\tau),k_1(\kappa_0,\kappa_1,\tau))\equiv\left(\frac{\kappa_1}{(s+b)\kappa_0},\tau\right).
 }
\end{equation}
Indeed, spherical curves in $\mathbb{E}^3$ are characterized by $\frac{k_1}{k_0}+[\frac{1}{k_1}(\frac{1}{k_0})']'=0$ \cite{DaSilvaMJM2018,Kuhnel2015}, which is equivalent to Eq. (\ref{eq::DEqFor4dRectCurves}) under the correspondence above. {We emphasize that by formal correspondence we mean that the existence of the associated curve is guaranteed by the fundamental theorem of curves in $\mathbb{E}^3$ and not necessarily by some sort of explicit construction. (The problem of finding an explicit construction is left as an open problem for the reader.)} 

The next theorem states that {the existence of correspondences as above} is a general feature {of} rectifying curves.

\begin{theorem}\label{TheoMapSpherCurveInRectCurve}
Let {$\{k_i\}_{i=0}^{m-1}$ and $\{\kappa_i\}_{i=0}^{m}$} denote the curvatures and torsion of regular curves in $\mathbb{E}^{m+1}$ and $\mathbb{E}^{m+2}$, respectively, then the correspondence 
\begin{equation}
{
    (\kappa_0,\dots,\kappa_m)\mapsto(k_0,k_1,\dots,k_{m-1})\equiv\left(\frac{\kappa_1}{(s+b)\kappa_0},\kappa_2,\dots,\kappa_m\right)
    }
\end{equation}
formally maps {rectifying curves in $\mathbb{E}^{m+2}$ into spherical curves in $\mathbb{E}^{m+1}$. Conversely, given a smooth function $f_0(s)>0$ and a constant $b$, then the correspondence} 
\begin{equation}
{
    (k_0,\dots,k_{m-1})\mapsto(\kappa_0,\kappa_1,\dots,\kappa_m)\equiv\Big(f_0,(s+b)k_0f_0,k_1,\dots,k_{m-1}\Big)
    }
\end{equation}
{formally maps spherical curves in $\mathbb{E}^{m+1}$ into rectifying curves in $\mathbb{E}^{m+2}$.} In addition, if $\{C_i\}$ and $\{A_i\}$ are respectively the coordinate functions of regular spherical and rectifying curves in $\mathbb{E}^{m+1}$ and $\mathbb{E}^{m+2}$ with respect to the their Frenet frames, then $(C_0,C_1)=(0,-\frac{1}{k_0})$ and $(A_0,A_1,A_2)=(s+b,0,\frac{(s+b)\kappa_0}{\kappa_1})$ and the remaining coordinate functions are related by
\begin{equation}
    (C_2,\dots,C_{m})\leftrightarrow (A_3,\dots,A_{m+1}).
\end{equation}
\end{theorem}
\begin{proof}
Let $\alpha$ be a spherical curve in $\mathbb{S}^{m}(r)\subset\mathbb{E}^{m+1}$ with coordinate functions {$\{C_i\}_{i=0}^{m+1}$ (with respect to the corresponding Frenet frame) and curvatures  $\{k_i\}_{i=0}^{m}$, where $k_0>0$ is the curvature and $k_{m}$ the torsion}. Since spherical curves can be seen as normal curves, we have $C_0=0$ and, therefore, according to Corollary \ref{Cor::FrenetSystemForCoord}, the remaining coordinates satisfy the system of equations
\begin{equation}
   \left\{
    \begin{array}{ccc}
         0 & = & 1+ k_0 C_{1}   \\
         C_{1}' & = & k_{1}C_{2} \\
         C_{i}' & = & -k_{i-1}C_{i-1}+k_{i}C_{i+1} \\
         C_{m}' & = & -k_m\, C_{m-1}\\
    \end{array}
    \right.,\,i\in \{2,\dots,m-1\}.
\end{equation}
On the other hand, the coordinate functions $\{A_i\}$ of a rectifying curve in $\alpha:I\to\mathbb{E}^{m+2}$ {(coordinates with respect to the Frenet frame of $\alpha$)} satisfy
\begin{equation}
     \left\{
    \begin{array}{ccc}
         A_{0}' & = & 1   \\
         0 & = & -\kappa_0 A_{0}+\kappa_{1}A_{2} \\
         A_{2}' & = & \kappa_{2} A_{3} \\
         A_{i}' & = & -\kappa_{i-1}A_{i-1}+\kappa_{i}A_{i+1} \\
         A_{m+1}' & = & -\tau A_{m}\\
    \end{array}
    \right.,\,i\in \{3,\dots,m\}.
\end{equation}
Comparison shows that under the associations $(C_2,\dots,C_{m})\leftrightarrow (A_3,\dots,A_{m+1})$ and {$(k_0,k_1,\dots,k_{m})\leftrightarrow\left(\frac{\kappa_1}{(s+b)\kappa_0},\kappa_2,\dots,\kappa_m\right)$, $\kappa_m=\tau$,} it is possible to establish a formal map between spherical and rectifying curves. {(Note that the arbitrary function $f_0>0$ appears because a single function $k_0$ is associated with two functions $\kappa_0$ and $\kappa_1$.)} Finally, the remaining coordinate function of the spherical curve is $C_1=-\frac{1}{k_0},$
while the two remaining coordinate functions of the rectifying curve are
$A_0=s+b$, {$A_1=0$,} and $A_2=\frac{\kappa_0}{\kappa_1}A_0=\frac{(s+b)\kappa_0}{\kappa_1}$. {The proof then follows as a consequence of the fundamental Lemma \ref{lemmaFrenetSystemForCoord}. Indeed, given the coordinates and curvatures functions we can integrate the corresponding Frenet systems of equations that will give us curves with the desired coordinates functions.}
\end{proof}

\begin{remark}
It is possible to write a single differential equation relating curvatures and torsion to characterize rectifying curves \cite{CambieTJM2016}. Under the correspondence given by the theorem above, we may then write a single differential equation characterizing spherical curves as well. Such an equation then generalizes the characterization of spherical curves in $\mathbb{E}^4$ and $\mathbb{E}^5$ given by the first named author and da Silva \cite{DaSilvaMJM2018} (see their Remark 2).
\end{remark}

Finally, there also exists a formal correspondence between $j$-rectifying curves in $\mathbb{E}^{m+2}$ and curves in $\mathbb{E}^{\,j}\times\mathbb{S}^{m-j}(r)$ for some $r>0$. Indeed, let $\alpha$ be a $j$-rectifying curve, its coordinate functions with respect to its Frenet frame satisfy the equations
$$
     \left\{
    \begin{array}{ccccc}
         A_{0}' & = & 1 + \kappa A_1 & & \\ [2pt]
         A_i' & = & -\kappa_{i-1}A_{i-1}+\kappa_iA_{i+1} &,& 1\leq i\leq j-2 \\ [2pt]
         A_{j-1}' & = & -\kappa_{j-1}A_{j-2} & & \\[2pt]
         A_{j+1}' & = & \kappa_{j+1} A_{j+2} & & \\ [2pt]
         A_{k}' & = & -\kappa_{k-1}A_{k-1}+\kappa_{k}A_{k+1} &,& j+2\leq k\leq m \\ [2pt]
         A_{m+1}' & = & -\tau A_{m} & &\\
    \end{array}
    \right.
$$
and $$0 = -\kappa_{j-1}A_{j-1}+ \kappa_{j}A_{j+1}.$$ 
The first $j$ functions $(A_0,A_1,\dots,A_{j-1})$ behave like the coordinates of a generic twisted curve in $\mathbb{E}^{\,j}$ with torsion $\kappa_{j-1}$, while the remaining coordinate functions together with $A_{0}=s+b$, i.e., $(s+b,A_{j+1},\dots,A_{m+1})$, behave like the coordinates of a rectifying curve in $\mathbb{E}^{m-j+2}$, which can be associated with a spherical curve in $\mathbb{E}^{m-j+1}$ according to Theorem \ref{TheoMapSpherCurveInRectCurve}. {We emphasize that, as in the proof of Theorem \ref{TheoMapSpherCurveInRectCurve}, the existence of curves with the desired properties comes as a consequence of the fundamental Lemma \ref{lemmaFrenetSystemForCoord}.} 

\section{Curves normal with respect to a rotation minimizing vector field}
\label{secRMrectifyingCurves}

In addition to the Frenet frame, we may equip a regular curve with the so-called rotation minimizing frames: we say that a unit $C^1$ vector field $\mathbf{V}$, normal to $\alpha'$, is \emph{rotation minimizing} (RM) if { $\mathbf{V}'(s)$ and $\mathbf{T}(s)$ are parallel} \cite{BishopMonthly1975}, i.e., if it is parallel transported with respect to the normal connection of the curve \cite{Etayo2016}. We now consider curves that always lie orthogonal to a rotation minimizing frame, a problem originally considered in 3d \cite{DaSilvaArXiv2017}, and show that this condition leads to plane and spherical curves. We say that a regular $C^2$ curve $\alpha$ is \emph{normal with respect to an RM vector field} $\mathbf{V}$ if $\langle\alpha(s)-p,\mathbf{V}(s)\rangle=0$, where $p$ is constant.

\begin{theorem}\label{ThrCharRMrectifying}
A curve is normal with respect to an RM field if and only if it is either a hyperplane or a spherical curve.
\end{theorem}
\begin{proof}
{Let $\mathbf{V}$ be an RM vector field along $\alpha$, i.e., $\mathbf{V}'=-\lambda\mathbf{T}$. We can build an RM frame $\{\mathbf{T},\mathbf{V}_1,\dots,\mathbf{V}_m,\mathbf{V}\}$ along $\alpha$ and containing $\mathbf{V}$. Indeed, in dimension 3, it is enough to take $\mathbf{V}_1=\mathbf{V}\times\mathbf{T}$. In the general case, take any set of vector fields $\mathbf{E}_1,\dots,\mathbf{E}_m$ orthogonal to both $\mathbf{T}$ and $\mathbf{V}$ whose equations of motion can be written as $\mathbf{E}_i'=k_{i0}\mathbf{T}+\sum_{j=1}^mk_{ij}\mathbf{E}_j$. (Note, $\langle\mathbf{E}_i',\mathbf{V}\rangle=0$, since $\mathbf{V}'=-\lambda\mathbf{T}$.) If we define new vector fields $\mathbf{V}_i=\sum_{j=1}^ma_{ij}\mathbf{E}_j$, $i\in\{1,\dots,m\}$, their derivatives along $\alpha$ are given by}
\begin{equation}
    \mathbf{V}_i' = \sum_{j=1}^m[a_{ij}'\mathbf{E}_j+a_{ij}k_{j0}\mathbf{T}+a_{ij}\sum_{\ell=1}^mk_{j\ell}\mathbf{E}_{\ell}]=\sum_{j=1}^ma_{ij}k_{j0}\mathbf{T}+\sum_{j=1}^m[a_{ij}'+\sum_{\ell=1}^ma_{i\ell}k_{\ell j}]\mathbf{E}_j.
\end{equation}
{Demanding all $\mathbf{V}_i$ to be RM is equivalent to finding a solution for the linear system of ordinary differential equations $A'=-AK$, where $A=(a_{ij})$ and $K=(k_{ij})$. Therefore, prescribing initial conditions $\mathbf{V}_1(s_0),\dots,\mathbf{V}_m(s_0)$ at $\alpha(s_0)$, there exist vector fields $\{\mathbf{V}_i\}$ orthogonal to $\mathbf{T}$ and $\mathbf{V}$ for \emph{all} values of $s$ such that each $\mathbf{V}_i$ minimizes rotation along $\alpha$. (Alternatively, we could have provided a shorter proof by appealing to the fact that minimizing rotation is a parallel transport with respect to the normal connection of $\alpha$ \cite{Etayo2016}.)}

{If $\alpha$ is normal with respect to an RM field $\mathbf{V}$, then we can write $\alpha(s)-p = A(s)\mathbf{T}(s)+A_1(s)\mathbf{V}_1(s)+\cdots+A_m(s)\mathbf{V}_m(s)$, where $\{\mathbf{T},\mathbf{V}_1,\dots,\mathbf{V}_m,\mathbf{V}\}$ is an RM frame along $\alpha$ (as constructed above) whose equations of motion are $\mathbf{T}'=\sum \kappa_i\mathbf{V}_i+\lambda\mathbf{V}$,  $\mathbf{V}_i'=-\kappa_i\mathbf{T}$, and $\mathbf{V}'=-\lambda\mathbf{T}$. Taking} the derivative of $\alpha-p$ gives
\begin{equation}
    \mathbf{T} = (A'-\sum_{i=1}^mA_i\kappa_i)\,\mathbf{T}+\sum_{i=1}^m(A_i'+A\kappa_i)\,\mathbf{V}_i+A\lambda\mathbf{V}.\label{EqTangOfRMrectifying}
\end{equation}
From the coordinate of $\mathbf{V}$, we deduce that $A\lambda=0$. If $\lambda=0$, then $\mathbf{V}$ is constant and, consequently, $\alpha-p$ lies in a hyperplane orthogonal to  $\mathbf{V}$. On the other hand, if $A=0$, then $\alpha$ is a normal curve, i.e., $\alpha$ is spherical, $\langle\alpha-p,\alpha-p\rangle=R^2\Leftrightarrow\langle\alpha-p,\mathbf{T}\rangle=0$. (Note, $\{A_i\}$ are all constant: from the coordinate of {$\mathbf{V}_i$} in Eq. \eqref{EqTangOfRMrectifying}, $A_i'=0$, and are related to the radius of the sphere by  $R^2=\langle\alpha-p,\alpha-p \rangle=\sum_{i=1}^mA_i^2$.)

Conversely, if $\alpha$ is spherical, $\alpha:I\to\mathbb{S}^{m+1}(p,R)$, the normal to the sphere, $\xi=\frac{1}{R}(\alpha-p)$, is an RM vector field. We may equip $\alpha$ with an RM frame $\{\mathbf{T},\mathbf{V}_1,\dots,\mathbf{V}_m,\xi\}$. Noticing that each $\mathbf{V}_i$ has to be tangent to the sphere, we deduce that $\alpha-p$ is normal to an RM vector field. The same reasoning applies to a hyperplane curve and the vector field normal to the plane. 
\end{proof}

\section*{Acknowledgements}
LCBdS would like to thank the financial support provided by the Mor\'a Miriam Rozen Gerber fellowship for Brazilian postdocs.


\begin{thebibliography}{99}

\bibitem{AltunkayaKJM2016}
B. Altunkaya, F. K. Aksoyak, L. Kula, and C. Aytekin, \textit{On rectifying slant helices in Euclidean 3-space}, Konuralp J. Math. \textbf{4} (2016), 17--24.

\bibitem{BishopMonthly1975}
R. L. Bishop, \textit{There is more than one way to frame a curve}, Am. Math. Mon. \textbf{82} (1975), 246--251.

\bibitem{CambieTJM2016}
S. Cambie, W. Goemans, and I. {Van den Bussche}, \textit{Rectifying curves in the $n$-dimensional Euclidean space}, Turk. J. Math. \textbf{40} (2016), 210--223.

\bibitem{ChenMonthly2003}
B.-Y. Chen, \textit{When does the position vector of a space curve always lie in its rectifying curve?}, Am. Math. Mon. \textbf{110} (2003), 147--152.

\bibitem{ChenTJM2017}
B.-Y. Chen, \textit{Rectifying curves and geodesics on a cone in the Euclidean 3-space,} Tamkang J. Math. \textbf{48} (2017), 209--214.

\bibitem{ChenBIMAS2005}
B.-Y. Chen and F. Dillen, \textit{Rectifying curves as centrodes and extremal curves}, Bull. Inst. Math. Acad. Sinica \textbf{33} (2005), 77--90.

\bibitem{AlghanemiFilomat2019}
S. Deshmukh, A. Z. Alghanemi, and R. T. Farouki, \textit{Space curves defined by curvature-torsion relations and associated helices}, Filomat \textbf{33} (2019), 4951--4966. 

\bibitem{DaSilvaArXiv2017}
L. C. B. Da Silva, \textit{Characterization of spherical and plane curves using rotation minimizing frames},  e-print arXiv:1706.01577. (To appear in Boletim da Sociedade Paranaense de Matem\'atica.)

\bibitem{DaSilvaMJM2018}
L. C. B. Da Silva and J. D. Da Silva, \textit{Characterization of curves that lie on a geodesic sphere or on a totally geodesic hypersurface in a hyperbolic space or in a sphere}, Mediterr. J. Math. \textbf{15} (2018), 70.

\bibitem{Etayo2016}
F. Etayo, \textit{Rotation minimizing vector fields and frames in {R}iemannian manifolds,} In: M. Castrill\'on--L\'opez, L. Hern\'andez--Encinas, P. Mart\'inez--Gadea, and M. E. Rosado--Mar\'ia (eds.) Geometry, Algebra and Applications: From Mechanics to Cryptography, Springer Proceedings in Mathematics and Statistics vol. 161, pp. 91--100. Springer, Berlin, 2016.

\bibitem{KimHMJ1993}
D. S. Kim, H. S. Chung, and K. H. Cho, \textit{Space curves satisfying $\tau/\kappa=as+b$}, Honam Math. J. \textbf{15} (1993), 1--9. 

\bibitem{Kuhnel2015}
W. K\"{u}hnel, \textit{Differential Geometry:
Curves - Surfaces - Manifolds,}  American Mathematical Society, Providence, 2015.

\bibitem{IlarslanTJM2008}
K. \.Ilarslan and E. Ne\v{s}ovi\'{c}, \textit{Some characterizations of rectifying curves in the Euclidean space $\mathbb{E}^4$}, Turk. J. Math. \textbf{32} (2008), 21--30.

\bibitem{IzumiyaTJM2004}
S. Izumiya and N. Takeuchi, \textit{New special curves and developable surfaces}, Turk. J. Math. \textbf{28} (2004), 153--163.

\bibitem{LucasJMAA2015}
P. Lucas and J. A. Ortega--Yag{\"u}es, \textit{Rectifying curves in the three-dimensional  sphere}, J. Math. Anal. Appl. \textbf{421} (2015), 1855--1868.

\bibitem{LucasMJM2016}
P. Lucas and J. A. Ortega--Yag{\"u}es, \textit{Rectifying curves in the three-dimensional  hyperbolic space},
 Mediterr. J. Math. \textbf{13} (2016), 2199--2214.

\bibitem{LucasBBMS2016}
P. Lucas and J. A. Ortega--Yag{\"u}es, \textit{Slant helices in the Euclidean 3-space revisited}, Bull. Belg. Math. Soc. Simon Stevin \textbf{23} (2016), 133--150.

\end{thebibliography}


\end{document}